\documentclass[11pt]{amsart}
\usepackage{graphicx}
\usepackage[active]{srcltx}
 \makeatletter
\renewcommand*\subjclass[2][2000]{%
  \def\@subjclass{#2}%
  \@ifundefined{subjclassname@#1}{%
    \ClassWarning{\@classname}{Unknown edition (#1) of Mathematics
      Subject Classification; using '1991'.}%
  }{%
    \@xp\let\@xp\subjclassname\csname subjclassname@#1\endcsname
  }%
}
 \makeatother
\usepackage{enumerate,amssymb,  mathrsfs,yhmath}

\newtheorem{theorem}{Theorem}[section]
\newtheorem{lemma}[theorem]{Lemma}
\newtheorem*{lemma*}{Lemma}
\newtheorem{proposition}[theorem]{Proposition}

\def\1ton{1,2,\ldots,n}
\def\det{{\rm det}}

\usepackage{amssymb}
\usepackage{amsthm}
\usepackage{mathrsfs, amsfonts, amsmath}
\usepackage{graphicx}

\newcommand{\bydef}{\stackrel{{\rm def}}{=\!\!=}}

\newcommand{\onto}{\xrightarrow[]{{}_{\!\!\textnormal{onto}\!\!}}}
\newcommand{\into}{\xrightarrow[]{{}_{\!\!\textnormal{into}\!\!}}}

\newcommand{\A}{\mathbb{A}}
\newcommand{\R}{\mathbb{R}}

\newcommand{\B}{\mathbb{B}}

\newcommand{\W}{\mathscr{W}}

\theoremstyle{definition}

\newtheorem{conjecture}[theorem]{Conjecture}

\theoremstyle{remark}
\newtheorem{remark}[theorem]{Remark}

\numberwithin{equation}{section}

\newcommand{\abs}[1]{\lvert#1\rvert}
 \DeclareMathOperator{\re}{Re}
\DeclareMathOperator{\im}{Im} 
\DeclareMathOperator{\Mod}{Mod}

\def\XXint#1#2#3{{\setbox0=\hbox{$#1{#2#3}{\int}$}
\vcenter{\hbox{$#2#3$}}\kern-.5\wd0}}

\def\ge{\geqslant}
\begin{document}

\title[Hyperelastic deformations and total combined energy]{Hyperelastic deformations and total combined energy of mappings between annuli} \subjclass{Primary 31A05;
Secondary 42B30 }


\author{David Kalaj}
\address{University of Montenegro, Faculty of Natural Sciences and
Mathematics, Cetinjski put b.b. 81000 Podgorica, Montenegro}
\email{davidk@ac.me}

\begin{abstract}

We consider the so called combined energy of a deformation between two concentric annuli and minimize it, provided that it keep order of the boundaries. It is an extension of the corresponding result of Euclidean energy. It is intrigue that, the minimizers are certain radial mappings and they exists if and only if the annulus on the image domain is not too thin, provided that the original annulus is fixed. This in turn implies a Nitsche type phenomenon. Next we consider the combined distortion and obtain certain related results which are dual to the results for combined energy, which also involve some Nitche type phenomenon.

{The main part of the paper is concerned with the total combined energy, a certain integral operator, defined as a convex linear combination of the combined energy and combined distortion, of diffeomorphisms between two concentric annuli $\A(1,r)$ and $\B(1,R)$. First we construct radial minimizers of total combined energy, then we prove that those radial minimizers are absolute minimizers on the class of all mappings between the annuli under certain constraint. This  extends the main result obtained by Iwaniec and Onninen in \cite{arma}.}

\end{abstract}  \maketitle


\section{Introduction}

Let $\A=\{z: r\le |z|\le R\}$ and $\B=\{w: R\le |w|\le R_\ast\}$ be rounded annuli in the complex plane $\mathbf{C}$.
We shall work with the homotopy class $\mathcal{F}(\A, \B)$ that consists of all orientation preserving
homeomorphisms $h : \A \onto \B$ between annuli, which keep the boundary
circles in the same order that belongs to the Sobolev class $\mathcal{W}^{1,1}$ together with its inverse which is usually denoted in this paper by $f:\B\to \A$. If $h\in \mathcal{F}(\A,\B)$, then $h_\circ(z)=\frac{1}{R} h\left(\frac{z}{r}\right)$ is an orientation preserving homeomorphism between annuli
$\A(1, R/r)$ and $\B(1, R_\ast/R)$. This is why without loss of generality, we will consider in this paper the following specific annuli $\A=\A(1,r)$ and $\B=\B(1,R)$.

The general law of hyperelasticity
asserts that there exists an energy integral
\begin{equation}\label{energyint}\mathcal{E} [h] = \int_{\A}E(z, h(z), Dh(z)) dxdy, \ \ z=x+iy\end{equation}
such that the elastic deformations have the smallest energy for $h\in \mathcal{F}(\A, \B)$. Here $E(z, h, Dh)$ is a function that is conformally coerced and polyconvex. Some additional regularity conditions
are also imposed. Under those conditions Onninen and Iwaniec in \cite[Theorem~1.2]{arma} have  established the existence and global
invertibility of the minimizers. In particular, those conditions are satisfied for the following list of functions $L$:
\begin{enumerate}
  \item $L_1(z,h,Dh)=|Dh(z)|^2$,
  \item $L_2(z,h,Dh)=\frac{|Dh(z)|^2}{\det\left[Dh\right](z) }$,
  \item $L_3(z,h,Dh)=\alpha|Dh(z)|^2+\beta \frac{|Dh(z)|^2}{\det\left[Dh\right](z) }$, $\alpha,\beta>0$.
  \end{enumerate}
 Here $$|A|^2
=\frac{1}{2}\mathrm{Tr}(A^TA)$$ is the square of mean
Hilbert-Schmidt norm of a matrix $A$.

 The case (1) refers to the conformal energy, the case (2) to the distortion and the case (3) refers to the so called total energy. First two cases have  been treated in the paper \cite{astala}. The same problem for non-Euclidean metrics has been treated in \cite{klondon}. The third case has been treated in detail in \cite{arma}. The several dimensional generalization of problem (1) for so-called $n-$harmonic mappings has been treated in detail in \cite{memoirs}, and rediscovered for radial metrics in \cite{arxivk}. Further $n-$dimensional version of the problem (3) but for radial mappings has been solved in \cite{kc}. It should be mentioned the fact that all those problems, have their root to the famous J. C. C. Nitsche conjecture  \cite{Nitsche}, for harmonic mappings, which have been solved by Iwaniec, Kovalev and Onninen  in \cite{nconj} after a number of subtle approaches in \cite{W}, \cite{L} and \cite{Ka}. All the mentioned papers and results
deal with properties of mappings between circular annuli (in the complex
plane or in Euclidean space). For a similar problem but for non-circular annuli
we refer to the paper \cite{invent} and its generalization in \cite{calculus}. For a related approach to the Riemann surface setting we refer to the paper \cite{isrg}.

  Polar coordinates
$z = x+iy=te^{i\theta}$, $1 < t < r$ and $0 \le \theta \le 2\pi $
are best suitable  for dealing with mappings of planar annuli. The radial (normal) and
angular (tangential) derivatives of $h : \A \to \B$ are defined by
$$h_N (z) = \frac{\partial h(te^{i\theta} )}{\partial t}, t = |z|$$
and
$$h_T (z)=\frac{1}{t}\frac{\partial h(te^{i\theta} )}{\partial \theta}, t = |z|$$
For a general map $h \in  \mathcal{F}(\A, \B)$ we have the formulas
\begin{equation}\label{dnt}|Dh|^2 = |h_N|^2 + |h_T|^2\end{equation} and \begin{equation}\label{jac}\det\left[Dh\right](z) = \im (h_T \overline{h_N} ) \le  |h_T | |h_N |.\end{equation}

If $h(z)=\rho(z)e^{i\Theta(z)}$, and $\nabla \rho$ and $\nabla \Theta$ are corresponding gradients then we have $$|Dh|^2=|\nabla\rho|^2+\rho^2|\nabla \Theta|^2$$

Let $\mathbf{a},\mathbf{b}>0$, $f\in\mathcal{F}(\A,\B)$ and $g\in \mathcal{F}(\B,\A)$ . In this paper we consider the following new concepts (and new functions): the combined energy ($L_4$ and $L'_5$), the combined distortion ($L_5$ and $L_4'$) and the total combined energy ($L_6$ and $L_6'$).
In this contexts, for $\mathbf{a}>0$ and $\mathbf{b}>0$ we define the squared norms of $Dh$ by $$\mathcal{D}[\mathbf{a},\mathbf{b}][h]=\mathbf{a}^2|h_N|^2+\mathbf{b}^2|h_T|^2.$$
and
$$\mathcal{D}'[\mathbf{a},\mathbf{b}][h]=\mathbf{a}^2\rho^2|\nabla\Theta|^2+\mathbf{b}^2|\nabla \rho|^2,$$ where $h=\rho e^{i\Theta}$. It should be noticed that, when $\mathbf{a}=\mathbf{b}=1$ then $\mathcal{D}[\mathbf{a},\mathbf{b}][h]=\mathcal{D}'[\mathbf{a},\mathbf{b}][h]$ and they coincides with $|Dh|^2$.
Then we consider the following list of functions
\begin{enumerate}
  \item[(4)] $L_4(z,h,Dh)=D[\mathbf{a},\mathbf{b}][h]$,
  \item[(5)] $L_5(z,h,Dh)=\frac{D[\mathbf{a},\mathbf{b}][h]}{\det\left[Dh\right](z) }$,
  \item[(6)] $L_6(z,h,Dh)=\alpha D[\mathbf{a},\mathbf{b}][h]+\beta \frac{D[\mathbf{a},\mathbf{b}][h]}{\det\left[Dh\right](z) }$, $\alpha,\beta>0$,
  \item[($4'$)] $L'_5(w,g,Dg)=\frac{D'[\mathbf{a},\mathbf{b}][g]}{\det\left[Dg\right](w) }$,
  \item[($5'$)] $L'_4(w,g,Dg)=D'[\mathbf{a},\mathbf{b}][g]$,
  \item[($6'$)] $L'_6(w,g,Dg)=\beta D'[\mathbf{a},\mathbf{b}][g]+\alpha \frac{D'[\mathbf{a},\mathbf{b}][g]}{\det\left[Dg\right](w) }$, $\alpha,\beta>0$.
  \end{enumerate}

All those functions are conformally coerced and polyconvex (a concept invented by Ball in \cite{ball}), and thus they satisfy the conditions of \cite[Theorem~1.2]{arma}. The mapping $h_\circ(z)=|z|^{\log R/\log r} e^{i\theta}$ has finite total combined energy and maps $\A$ onto $\B$. From \cite[Theorem~1.2]{arma} we infer that there exists a deformation $h_\diamond$ in the same homotopy class as $h_\circ$, having smallest energy defined in \eqref{energyint}.

Here we are concerned with the construction of $h_\diamond$ for $L_4,L_5,L_6$ and their related functions $L_4'$, $L_5'$ and $L_6'$.
Before we go further let us state the following proposition which follows from related change rule results obtained in \cite{heko} and Lemma~\ref{bel}  below:

\begin{proposition}
For $i=4,5,5$ and $h\in \mathcal{F}(\A,\B)$ we have
\begin{equation}\label{dri}\int_{\A}L_i(z,h,Dh)dz=\int_{\B}L'_i(w,g,Dg)dw,\end{equation} where $g=h^{-1}$.
\end{proposition}

The key tools in obtaining an extremal deformation
$h : \A \onto \B$, regardless of its boundary values, are the free Lagrangians. Finding
suitable free Lagrangians and using is
a challenging and intrigue issue in this paper. We have done it here for the so-called total combined harmonic energy
and a pair of annuli in the plane. In fact this challenging problem illustrates rather
clearly the strength of the concept of free Lagrangians.

In this context a free Lagrangian
refers to a differential $2$-form $$L(z, h, Dh) dz\bydef L(z, h, Dh) dx\wedge dy$$ whose integral over $\A$ does not depend
on a particular choice of the mapping $h \in \mathcal{F}(\A, \B)$.

Together with this introduction, the paper contains four more sections. In the third section we define the combined energy and find the radial extremals. Then we optimize the energy under the so-called J.C.C. Nitsche condition (Theorem~\ref{l4}). In the fourth section, we define the combined distortion and calculate the radial extremals. Then we prove the corresponding theorem that says that radial minimizers are those who minimize the combined distortion, provided that the annuli satisfy J.C.C. Nitsche condition (Theorem~\ref{delcomb}).

In the last section we prove the main results of the paper. First we define the total combined energy, then we obtain radial extremals. It is important to say that in this case we do not have any constraint on annuli. We finish this section by proving the main result of the paper which roughly speaking says that, if the ratio between the "combined" constants $\mathbf{a}$ and $\mathbf{b}$ is close to $1$, then the radial minimizers of total combined energy are absolute minimizers of total combined energy (Theorem~\ref{comtotal}).

 In order to formulate a corollary of our main result Theorem~\ref{comtotal} and of Proposition~\ref{propo} let  $$\mathfrak{E}_{\mathbf{c}}[h]\bydef\int_{\A}L_6(z,h,Dh)dz, \text{  where  } \mathbf{c}=\frac{\mathbf{a}}{\mathbf{b}}.$$
\begin{theorem}\label{comtotal}
For fixed pair of annuli $\A$ and $\B$, let $\mathbf{m}=\frac{\Mod(\A)}{\Mod(\B)}$. Then there is a constant $\varepsilon=\varepsilon(\mathbf{m})$ which is greater than $0$ for $\mathbf{m}\neq 1$ so that if $0\le |\mathbf{c}-1|\le\varepsilon$, then the total combined energy integral $\mathfrak{E}_{\mathbf{c}}: \mathcal{F}(\A,\B)\to \mathbf{R}$,  attains its minimum  for a radial mapping $h_\circ$.  The minimizer is unique up to a rotation.
\end{theorem}
Thus our result extends  \cite[Theorem~1.4]{arma} which is the main result in \cite{arma} ($\mathbf{c}=1$).
\section{Free Lagrangians}
As  Iwaniec and Onninen did in their papers \cite{arma, memoirs, annalen} we consider the following free Lagrangians.

a)  A function in $t= |z|$;
\begin{equation}\label{functionx}L(z, h, Dh) dz = M(t) dz \end{equation}
Thus, for all $h \in  \mathcal{F}(\A, \B)$ we have

\begin{equation}\label{simpli}\int_\A L(z, h, Dh) dz =
\int_\A M(|z|) dz=2\pi\int_r^R t M(t) dt.\end{equation}

b)  Pullback of a form in $\B$ via a given mapping $h \in  \mathcal{F}(\A, \B)$;
\begin{equation}\label{pullback}L(z, h, Dh) dx\wedge dy = N(|h|) J (z, h) dz, \text{where} \ \ \ N \in  L^ 1(\R, \R_\ast) \end{equation}
Thus, for all $h \in  \mathcal{F}(\A, \B)$ we have

\begin{equation}\label{easi0}\int_\A L(z, h, Dh) dz =\int_{\B}
N(|w|) du\wedge dv = 2\pi \int_{1}^{R}
sN(s )  ds  \end{equation}

c)  A radial free Lagrangian
\begin{equation}\label{radial}L(z, h, Dh) dz = A (|h|)\frac{|h|_N}{ |z|}
dz\end{equation} where $A \in  L^   1(1, R)$
Thus, for all $h \in  \mathcal{F}(\A, \B)$ we have

\begin{equation}\label{easi}\int_\A L(z, h, Dh) dz = 2\pi \int_1^R A(|h|)\frac{\partial |h|}{\partial \rho} d\rho  = 2\pi\int_{1}^{R} A(s)ds  \end{equation}
d)  An angular free Lagrangian
\begin{equation}\label{angular}L(z, h, Dh)dz  = B (|z|)\im\left[\frac{h_T}{h}\right]dz ,\end{equation} where $B \in  L^1(R, R_\ast)$.
Thus, for all $h \in \mathcal{F}(\A, \B)$ we have

\begin{equation}\label{easi1}\int_A
L(z, h, Dh) dz =\int_1^r
\frac{B(t)}{t}
\int_{|z|=t}
\left(\frac{\partial \mathrm{Arg}\, h}{\partial \theta}
d\theta\right) dt =\int_1^r
\frac{B(t)}{t}
dt.\end{equation}
e) Every smooth function $\mathcal{A}(t,s)$ of two variables $1\le t\le r$, $1\le s\le R$ produces the following free Lagrangian
$$E(z,h, Dh)=\frac{\partial_t\mathcal{A}(|z|, |h(z)|)}{t}=\frac{\mathcal{A}_t(|z|, |h(z)|)+\mathcal{A}_s(|z|, |h(z)||h|_N)}{t},$$ where $t=|z|$ and $s=|h(z)|$.

Namely it can be easy shown that (\cite[Eq.~212]{arma})  $$\mathcal{E}[h]\bydef\int_{\A} E(z,h, Dh)dxdy=2\pi [\mathcal{A}(r,R)-\mathcal{A}(1,1)].$$

The idea behind our use of these free Lagrangians as in \cite{arma,memoirs,annalen} is to establish a general subgradient
type inequality for the integrand with two independent parameters $t \in (1, r)$ and
$s\in (1, R)$,
$$\mathbf{b}^2 |h_T|^2+\mathbf{a}^2 |h_N|^2\ge  X(s)\frac{|h|_N}{t}
+ Y(t)
\frac{|h|}{s }\im\frac{h_T}{h} + Z(s)\det\left[ Dh\right] + W(t) $$
where the coefficients $U$ and $V$ are functions in the variable $1 <s< R$, while $Z$ and
$W$ are functions in the variable $1 < t < r$. We will choose $t = |z|$ and $s = |h(z)|$ to
obtain free Lagrangians in the right hand side. We will find
coefficients  to ensure equality for a solution to radial Euler- Lagrange equation,
which we expect to be a unique minimizer up to rotations of the annuli. Finding such
coefficients requires deep analysis of this problem.

\section{Combined energy}
Assume as before  that $\A=\{z: 1\le |z|\le r\}$ and $\B=\{w: 1\le |w|\le R\}$.
Assume that $\mathbf{b}>0$ and $\mathbf{a}>0$ and let $h\in \mathcal{F}(\A, \B)$ and consider the energy $$\mathcal{E}[h]=\mathcal{E}[\mathbf{a},\mathbf{b}][h]\bydef \int_{\A}(\mathbf{a}^2 |h_N|^2+\mathbf{b}^2|h_T|^2) dxdy.$$

\subsection{Energy of radial mappings}
If $h(z)=H(t)e^{i\theta+i\varphi}$, $z=t e^{i\theta}$, then we say that $h$ is a radial mapping, where $H$ is a positive real function. Then $$\mathcal{E}[h]=2\pi \int_{1}^{r} t\left(\mathbf{a}^2\dot H^2(t)+\mathbf{b}^2\frac{H^2(t)}{t^2}\right)dt =\int_1^r\mathfrak{L}(t,H,\dot H)dt.$$
Then the Euler--Lagrange equation  $$\partial_H\mathfrak{L}(H,\dot H, t)=\frac{d}{d{t}}\partial_{\dot H} \mathfrak{L}(H,\dot H, t)$$ reduces to the differential equation $$H''(t)- \frac{\mathbf{b}^2 H(t)-\mathbf{a}^2 t H'(t)}{\mathbf{a}^2 t^2}=0.$$
The general solution is $$H(t)=c_1 \cosh\left[\frac{{\mathbf{b}}\log t}{{\mathbf{a}}}\right]+c_2 \sinh\left[\frac{{\mathbf{b}}\log t}{{\mathbf{a}}}\right].$$
Assuming that $H(1)=1$, we have  $$H(t)= \cosh\left[\frac{{\mathbf{b}}\log t}{{\mathbf{a}}}\right]+\mu \sinh\left[\frac{{\mathbf{b}}\log t}{{\mathbf{a}}}\right],\ \ \ \mu\in \mathbf{R},$$ which can be written as
\begin{equation}\label{Ht} H(t)=\frac{1}{2} t^{-1/\mathbf{c}} \left(1-\mu+(1+\mu) t^{2 /\mathbf{c}}\right), \ \ \ \ \mathbf{c}=\frac{\mathbf{a}}{\mathbf{b}}.\end{equation}
Then $$H'(t)=\frac{1}{2\mathbf{c}}  t^{-1-/\mathbf{c}} \left(-1+\mu+(1+\mu) t^{2 /\mathbf{c}}\right).$$
Thus $H'(1)=\mu /\mathbf{c}\ge 0$ if and only if $\mu\ge 0$, so if we want to map the interval $[1,r]$ onto $[1,R]$ by an increasing diffeomorphism $H$, then $\mu \ge 0$ and thus
\begin{equation}\label{Nitchegeneral}R \ge \cosh\left[{\mathbf{c}}{\log r}{}\right] =\frac{1+r^{2/\mathbf{c}}}{2r^{1/\mathbf{c}}}(\mathrm{Nitsche\ \ type\ \ inequality}).\end{equation}  Moreover if $H(r)=R$, then $$\mu = \frac{1+r^{2 /\mathbf{c}}-2 r^{1/\mathbf{c}} R}{1-r^{2 /\mathbf{c}}}$$
Then \begin{equation}\label{elastic}\mathbf{c}\ge \frac{H(t)}{t H'(t)} \text{  if and only if } \mu\ge 1,\end{equation} and

\begin{equation}\label{inelastic}\mathbf{c}<\frac{H(t)}{t H'(t)}< 0 \text{  if and only if } 0<\mu< 1,\end{equation}
For $$h_\diamond(z)=H(t)e^{i\theta}=\frac{\left(1-\mu+(1+\mu) t^{2 /\mathbf{c}}\right)}{2t^{1/\mathbf{c}}}  e^{i\theta},$$
by direct computation we obtain the relation
$$\mathcal{E}[h_\diamond]=\frac{2\mathbf{b}\mathbf{a} \pi  \left(1-4 r^{1/\mathbf{c}} R+R^2+r^{2 /\mathbf{c}} \left(1+R^2\right)\right)}{ \left(r^{2 /\mathbf{c}}-1\right)}.$$

The inverse mapping $F$ of the mapping $H$ defined in \eqref{Ht} is give by the formula

\begin{equation}\label{Fs}
F(s)=\left(\frac{s+\sqrt{-1+\mu^2+s^2}}{1+\mu}\right)^{\frac{\mathbf{a}}{\mathbf{b}}}.
\end{equation}

\begin{theorem}\label{l4}
Under condition \eqref{Nitchegeneral}, the combined energy integral $\mathcal{E}: \mathcal{F}(\A,\B)\to \mathbf{R}$,  attains its minimum  for a radial mapping $h_\diamond$. The minimizer is unique up to a rotation.
\end{theorem}

\begin{proof}
We divide the proof into two cases
\subsection{Elastic case: $\mu\ge 1$}

We have the following general inequality

\begin{equation}\label{general}\mathbf{a}^2 |h_N|^2+\mathbf{b}^2|h_T|^2\ge (\mathbf{a}^2-\mathbf{b}^2 a^2)|h_N|^2 +(\mathbf{b}^2-\mathbf{a}^2 b^2) |h_T|^2 + 2 a b \mathbf{a}^2\mathbf{b}^2  |h_T| |h_N|\end{equation}

\begin{equation}\label{general1}\frac{\mathbf{b}^2 |h_N|^2+\mathbf{a}^2|h_T|^2}{J(h,z)}
\ge \frac{(\mathbf{b}^2-\mathbf{a}^2 a_\ast^2)|h_N|^2 +(\mathbf{a}^2-\mathbf{b}^2 b_\ast^2) |h_T|^2 + 2 a_\ast b_\ast \mathbf{a}^2\mathbf{b}^2  |h_T| |h_N|}{J(h,z)}\end{equation}

Then we take $b=\mathbf{b}/\mathbf{a}$ and it reduces to the following simple inequality

\begin{equation}\label{first}\mathbf{a}^2 |h_N|^2+\mathbf{b}^2|h_T|^2\ge (\mathbf{a}^2-\mathbf{b}^2 a^2)|h_N|^2 +2 b \mathbf{b}^2  |h_T| |h_N|,\end{equation} with the equality if and only if \begin{equation}\label{simpli11eq1}a|h_N|=|h_T|.\end{equation}
Further
\begin{equation}\label{ess}\mathbf{a}^2 |h_N|^2+\mathbf{b}^2|h_T|^2\ge\Lambda:\bydef [2(\mathbf{a}^2 -\mathbf{b}^2 a^2)A ]{|h_N|}+ 2 b \mathbf{b}^2\det\left[ Dh\right] -{(\mathbf{a}^2 -\mathbf{b}^2 a^2) A^2}.\end{equation} The equality is attained in \eqref{ess} if

\begin{equation}\label{eqat}
A=|h_N| \ \ \  \ \text{and} \ \ \  h  \ \ \ \text{is radial}.
\end{equation}
Let  \begin{equation}\label{ts} \text{$t=|z|$ and $s=|h(z)|$,}\end{equation}
\begin{equation}\label{b}a=\frac{sF'(s)}{F(s)}\end{equation}
\begin{equation}\label{bast}a_\ast=\frac{H(t)}{tH'(t)}\end{equation} and
\begin{equation}\label{bcapital}A=H'(t)\sqrt{\frac{\mathbf{a}^2 -\mathbf{b}^2{a^2_\ast}}{\mathbf{a}^2-\mathbf{b}^2a^2}}.\end{equation}
Since $$|h|_N=\frac{\left<h_N, h\right>}{|h|},$$ it follows that $|h_N|\ge |h|_N$ and thus
\begin{equation}\label{lambda}\Lambda\ge 2\left[(\mathbf{a}^2 -\mathbf{b}^2a^2)A \right]{|h|_N}+2\mathbf{b}^2\frac{sF'(s)}{F(s)}\det\left[ Dh\right] -(\mathbf{a}^2-\mathbf{b}^2 a^2)A^2.\end{equation}

Again the equality statement in  \eqref{lambda} is achieved if $h$ satisfies \eqref{eqat}. Notice that $s=|h(z)|$, and for the radial mapping $h(z)=H(t)e^{i\theta}$, $s=H(t)$.
Further
\begin{equation}\label{ss}\begin{split}2t(\mathbf{a}^2 -\mathbf{b}^2a^2)A &=2t \sqrt{\mathbf{a}^2-\mathbf{b}^2{a^2_\ast}}\sqrt{\mathbf{a}^2-\mathbf{b}^2 a^2} \dot H(t)\\&=2\sqrt{\mu^2-1}\sqrt{\mathbf{a}^2-\mathbf{b}^2{a^2}}\\&=2 \mathbf{b} \sqrt{\mu^2-1}:=\Upsilon(s)\end{split}\end{equation} and

\begin{equation}\label{tt}
-(\mathbf{a}^2-\mathbf{b}^2 a^2)A^2=-\dot H^2(t)(\mathbf{a}^2-\mathbf{b}^2 a^2_\ast)=\frac{\mathbf{b}^2 \left(1-\mu^2\right)}{t^2}=:\Gamma(t).\end{equation}

So for $t=|z|$ and $s=|h(z)|$, by using the relations\eqref{first}--\eqref{tt} we obtain \begin{equation}\begin{split}
\mathcal{E}[h]&=\int_{\A} (\mathbf{a}^2 |h_N|^2+\mathbf{b}^2|h_T|^2)dz\\&\ge \int_{\A} \left[2\left[(\mathbf{a}^2 -\mathbf{b}^2a^2)A \right]{|h|_N}+2\mathbf{b}^2\frac{sF'(s)}{F(s)}\det\left[ Dh\right] -(\mathbf{a}^2-\mathbf{b}^2 a^2)A^2\right] dz
\\&=\int_{\A}  \Upsilon(s) \frac{|h|_N}{t} dz+2\mathbf{b}^2\int_{\A}  \frac{|h(z)|F'(|h(z)|)}{F(|h(z)|)}\det\left[ Dh\right] dz+\int_{\A}\Gamma(t)dz.
\end{split}
\end{equation}
By making use of the formulas \eqref{easi0},\eqref{easi}, \eqref{easi1} and \eqref{simpli} we obtain

\begin{equation}\label{upsi}\mathcal{E}[h]\ge 2\pi\int_{1}^{R}  \Upsilon(s) ds+2\mathbf{b}^2\cdot 2\pi \int_{1}^{R}   \frac{s^2F'(s)}{F(s)}ds+2\pi \int_{1}^rt\Gamma(t)dt. \end{equation}
On the other hand, straightforward backward analysis shows that the combined energy of the mapping $h_\diamond(z)= H(t)e^{i\theta}$ is equal to the righhand side od \eqref{upsi}.

\subsection{Non-elastic case: $0\le \mu< 1$}
For a fixed diffeomorphism $h$ between two annuli $\A$ and $\B$, and $z\in \A$ and $h(z)\in \B$, as before let   $t=|z|$ and $s=|h(z)|$. If we put $a=\mathbf{a}/\mathbf{b}$ in \eqref{general}, then we have the following inequality
\begin{equation}\label{simpli11}\mathbf{a}^2 |h_N|^2+\mathbf{b}^2|h_T|^2\ge (\mathbf{b}^2-\mathbf{a}^2 b^2)|h_T|^2 +2 b \mathbf{a}^2  |h_T| |h_N|,\end{equation} where the equality hold  if and only if \begin{equation}\label{simpli11eq}b|h_T|=|h_N|.\end{equation}

By \eqref{jac},  $\det[Dh]\le |h_T| |h_N|$, with equality if and only if $h$ is radial, and therefore we continue with the following simple inequality

\begin{equation}\label{simpli1}\mathbf{a}^2 |h_N|^2+\mathbf{b}^2|h_T|^2\ge \mathcal{L}:\bydef [2(\mathbf{b}^2 -\mathbf{a}^2 b^2)B ]{|h_T|}+ 2 b \mathbf{a}^2\det\left[ Dh\right] -{(\mathbf{b}^2 -\mathbf{a}^2 b^2) B^2}.\end{equation} The equality in \eqref{simpli1} is attained if and only if \begin{equation}A=|h_T|\ \ \ \ \text{and}\ \ \ \ h \ \ \text{is radial}.\end{equation}

Choose now
\begin{equation}\label{a1}b=\frac{F(s)}{sF'(s)}\end{equation}
%

\begin{equation}\label{AA}B=\frac{s}{t}.\end{equation}
Since $\im\left[\frac{h_T}{h}\right]\le\left[\frac{|h_T|}{|h|}\right]$ we infer that $$|h_T|\ge |h|\im\left[\frac{h_T}{h}\right]$$ and thus
\begin{equation}\label{help}
\mathcal{L}\ge [2(\mathbf{b}^2 -\mathbf{a}^2 b^2)B ]{|h|\im\left[\frac{h_T}{h}\right]}+ 2 \mathbf{a}^2\frac{F(s)}{sF'(s)} \det\left[ Dh\right] -{(\mathbf{b}^2 -\mathbf{a}^2 b^2) B^2}.
\end{equation}
The equality is attained in \eqref{help} if $h$ is radial, and if the condition \eqref{simpli11eq} is satisfied.
Now we easily obtain that  $$b=\frac{s\dot F(s)}{F(s)}=\frac{\mathbf{b} \sqrt{-1+\mu^2+s^2}}{\mathbf{a} s}.$$ Further we get
\begin{equation}\label{pr}\begin{split}-{(\mathbf{b}^2 -\mathbf{a}^2 b^2) B^2}&=\mathcal{M}(t)\bydef -\frac{\mathbf{b}^2 \left(1-\mu^2\right)}{t^2}\end{split}\end{equation}
and from \eqref{Ht} we have
$$ b_\ast =\frac{t\dot H(t)}{H(t)}=\frac{\left(-1+\mu+(1+\mu) t^{2/ \mathbf{c}}\right)}{\mathbf{c} \left(1-\mu+(1+\mu) t^{2/ \mathbf{c}}\right)}. $$
 Thus
\begin{equation}\label{pri}\begin{split}2s(\mathbf{b}^2 -\mathbf{a}^2 b^2)B=\mathcal{N}(t)\bydef\frac{2 \mathbf{b} (\mathbf{c}^2-1)}{\mathbf{c}^2 t}.\end{split}\end{equation}
 So for $t=|z|$ and $s=|h(z)|$, by using the relations\eqref{simpli11}--\eqref{pri} we obtain \begin{equation}\begin{split}
\mathcal{E}[h]&=\int_{\A} (\mathbf{a}^2 |h_N|^2+\mathbf{b}^2|h_T|^2)dz\\&\ge \int_{\A} \left[[2(\mathbf{b}^2 -\mathbf{a}^2 b^2)B ]{{|h|\im\left[\frac{h_T}{h}\right]}}+ 2\mathbf{a}^2 a \det\left[ Dh\right] -{(\mathbf{b}^2 -\mathbf{a}^2 b^2) B^2}\right] dz
\\&=\int_{\A}  \mathcal{N}(t) \im\left[\frac{h_T}{h}\right] dz+2\mathbf{a}^2\int_{\A} \frac{F(s)}{sF'(s)} \det\left[ Dh\right] dz+\int_{\A}\mathcal{M}(t)dz.
\end{split}
\end{equation}

By making use of the formulas \eqref{easi0},\eqref{easi}, \eqref{easi1} and \eqref{simpli} we obtain

\begin{equation}\label{upsi2}\mathcal{E}[h]\ge 2\pi\int_{1}^{r}  \frac{\mathcal{N}(t)}{t} dt+2\mathbf{a}^2\cdot 2\pi \int_{1}^{R}   \frac{F(s)}{F'(s)}ds+2\pi \int_{1}^rt\mathcal{M}(t)dt. \end{equation}

If $h(z)=h_\diamond(z)=H(t)e^{i\theta}$, then $$\im\left[\frac{h_T}{h}\right]=\im\left[\frac{i H(t)e^{i\theta}}{t H(t)e^{i\theta}} \right]=\frac{1}{t},$$ and
$|h|=H$ and $$\det\left[ Dh\right]=\im\left[h_T \overline{h_N}\right]=\frac{\dot H(t)H(t)}{t}=|h_T|\cdot |h_N|.$$  It follows that
\[\begin{split}\mathcal{L}&=[2(\mathbf{b}^2 -\mathbf{a}^2 b^2)B ]{{|h|\im\left[\frac{h_T}{h}\right]}}+ 2\mathbf{a}^2 b \det\left[ Dh\right] -{(\mathbf{b}^2 -\mathbf{a}^2 b^2) B^2}\\&=\mathbf{a}^2 |h_N|^2+\mathbf{b}^2|h_T|^2.\end{split}\] Thus
$\mathfrak{E}[h]\ge \mathfrak{E}[h_\diamond].$
\end{proof}

\section{Combined distortion}
For $\mathbf{a}, \mathbf{b}>0$ and $h\in \mathcal{F}(\A,\B)$ we define the combined distortion by \begin{equation}\label{gini}\mathcal{K}[\mathbf{b},\mathbf{a}][h]\bydef \int_{\A} \frac{\mathbf{a}^2\rho^2  |\nabla \Theta|^2+\mathbf{b}^2 |\nabla \rho  |^2}{J(h,z)}dz,\end{equation} where  $h(z)=\rho e^{i\Theta}$.

Before we continue, observe that if $z=te^{i \theta}$ and  $h(z)=H(t)e^{i\theta}$ then \begin{equation}\label{conte}\frac{\mathbf{b}^2 |\nabla \rho  |^2+\mathbf{a}^2\rho^2  |\nabla \Theta|^2}{\re\left[h_N\overline{h_T}\right]}=\frac{\mathbf{b}^2\dot H^2(t)+\mathbf{a}^2\frac{H^2(t)}{t^2}}{\dot H \frac{H}{t}}.\end{equation}
Further, from Lemma~\ref{bel}, we have
\begin{equation}\label{invers}\mathcal{K}[\mathbf{b},\mathbf{a}][h]=\mathcal{E}[\mathbf{a},\mathbf{b}][h^{-1}].\end{equation}
So the subintegral expression in \eqref{gini} reduces, up to a multiplicative constant, to $$\Lambda(H,\dot H, t)\bydef t^2\frac{\mathbf{b}^2\dot H^2(t)+\mathbf{b}^2\frac{H^2(t)}{t^2}}{\dot H {H}}.$$ Then the Euler--Lagrange is
 $$\partial_H\Lambda(H,\dot H, t)=\frac{d}{d{t}}\partial_{\dot H} \Lambda(H,\dot H, t)$$ which reduces to the equation
 \begin{equation}\label{hf}
 H''[t]=-\frac{H'[t]^2 \left(-\mathbf{a}^2 H[t]+\mathbf{b}^2 t H'[t]\right)}{\mathbf{a}^2 H[t]^2}.
 \end{equation}
Then in view of the previous section and \eqref{invers} we obtain that the sufficient and necessary condition to exist a diffeomorphism $H:[1,r]\to [1, R]$ is the following inequality
\begin{equation}\label{Nitchegeneral1}r \ge \cosh\left[\mathbf{c}{\log R}\right] =\frac{1+R^{2/\mathbf{c}}}{2R^{1/\mathbf{c}}}(\mathrm{Nitsche\ \ type\ \ inequality}),\end{equation} which is dual to \eqref{Nitchegeneral}.

Moreover, as in \eqref{Fs} we obtain that the solution of \eqref{hf} that maps $[1,r]$ onto $[1,R]$ has its representaition
\begin{equation}\label{FsHt}
H(t)=\left(\frac{t+\sqrt{-1+\nu^2+t^2}}{1+\nu}\right)^{\frac{\mathbf{a}}{\mathbf{b}}},
\end{equation}
where
$$\nu = \frac{1+r^{2 /\mathbf{c}}-2 R^{1/\mathbf{c}} r}{1-R^{2 /\mathbf{c}}}.$$

Now we put $h_\diamond=H(t)e^{i\theta}$.

Now we formulate the related result for distortion:
\begin{theorem}\label{delcomb}
Under condition \eqref{Nitchegeneral1}, the combined distortion integral $\mathcal{K}: \mathcal{F}(\A,\B)\to \mathbf{R}$,  attains its minimum  for a radial mapping $h_\diamond$. The minimizer is unique up to a rotation.
\end{theorem}

We begin by the following lemma
\begin{lemma}\label{bel}
Assume that $h$ is a deformation of $\A$ onto $\B$, and let $g=h^{-1}$. If
\begin{equation}\label{gin}\mathcal{E}[\mathbf{a},\mathbf{b}][g]= \int_{\B} \mathbf{a}^2 |g_N|^2+\mathbf{b}^2 |g_T|^2dw,\end{equation} then

\begin{equation}\label{gin1} \mathcal{E}[\mathbf{a},\mathbf{b}][g]=\mathcal{K}[\mathbf{b},\mathbf{a}][h]\end{equation}
where $\mathcal{K}$ is defined in \eqref{gini}.
\end{lemma}
\begin{proof}
We begin by the equality $$g(\rho e^{i\Theta})=z=t e^{i\theta}.$$ By differentiating we obtain

$$ \partial_\rho g(\rho e^{i\Theta}) \rho_t + \partial_\Theta g(\rho e^{i\Theta}) \Theta_t = e^{i\theta}$$ and

$$ \partial_\rho g(\rho e^{i\Theta}) \rho_\theta + \partial_\Theta g(\rho e^{i\Theta}) \Theta_\theta = it e^{i\theta}.$$
By solving in $ \partial_\rho g$ and $\partial g_\Theta$, we get

\begin{equation}\label{one}g_N=\partial_\rho g(\rho e^{i\Theta})=\frac{e^{i\theta}(\Theta_\theta-it\Theta_t ) }{\rho_t\Theta_\theta-\rho_\theta\Theta_t}\end{equation} and

\begin{equation}\label{due}g_T=\frac{\partial_\Theta g (\rho e^{i\Theta})}{\rho}=\frac{e^{i\theta}(\rho_\theta-it\rho_t ) }{\rho(\rho_t\Theta_\theta-\rho_\theta\Theta_t)}.\end{equation}

If $h(z)=\rho e^{i\Theta}=\rho\cos \Theta+i \rho\sin\Theta$, then
\[\begin{split}J(z,h)&=(\rho \cos(\Theta))_x (\rho \sin(\Theta))_y-(\rho \cos(\Theta))_y (\rho \sin(\Theta))_x\\&=\rho(\Theta_x \rho_y-\Theta_y\rho_x)
\\&= \frac{\rho}{t}\left(\rho_t\Theta_\theta-\rho_\theta\Theta_t\right).
\end{split}\]
Thus \begin{equation}\label{tre}\rho_t\Theta_\theta-\rho_\theta\Theta_t= \frac{tJ(h,z)}{\rho}.\end{equation}

Further \begin{equation}\label{ttt}|\Theta_\theta-it\Theta_t|=t|\nabla \Theta|, \ \ \ \ |\rho_\theta-it\rho_t |=t|\nabla \rho|.\end{equation}
By applying the change of variables $w=h(z)$ in \eqref{gin}, by using \eqref{one}, \eqref{due}, \eqref{tre} and \eqref{ttt} we obtain \eqref{gin1}.

\end{proof}

Further we have
\begin{lemma} Let $h=\rho e^{i\Theta}$. Then
$$|\nabla \rho|\ge |h|_N$$ and $$|\nabla \Theta|\ge \abs{\im\left[\frac{h_T}{h}\right]}.$$
\end{lemma}
\begin{proof}
The first equation follows from the following formula $$|h|_N= \left<\nabla |h|, \frac{z}{|z|}\right>.$$
Further if $z=te^{i\theta}$ then $$h_T=\frac{1}{t} e^{i\Theta}( \rho_\theta + i\Theta_\theta).$$
 So $$\abs{\im\left[\frac{h_T}{h}\right]}=\abs{\im \left[\frac{\frac{1}{t} e^{i\Theta}( \rho_\theta + i\rho\Theta_\theta)}{\rho e^{i\Theta}}\right]}=\frac{|\Theta_\theta|}{t}.$$ On the other hand $$|\nabla \Theta|^2=\Theta_x^2+\Theta_y^2.$$

 Since $$\Theta_\theta=\Theta_x x_\theta+\Theta_y y_\theta=t(-\sin\theta \Theta_x +\cos \theta \Theta_x),$$ by using the Cauchy -Schwarz inequality the second inequality follows.

 \end{proof}

\subsection{Proof of Theorem~\ref{delcomb}}
We have the following pointwise trivial estimate
\[\begin{split}{K}[h]&=\frac{\mathbf{b}^2|\nabla \rho|^2+\mathbf{a}^2 \rho^2|\nabla \theta|^2}{\im[h_T \overline{h_N}]}
\\&\ge  \frac{\mathbf{b}^2 |h|_N^2+\mathbf{a}^2\rho^2 \abs{\im\left[\frac{h_T}{h}\right]}^2}{\im[h_T \overline{h_N}]}
\\&\ge \frac{(\mathbf{b}^2 -a_\ast^2\mathbf{a}^2)|h|_N^2+(\mathbf{a}^2-b_\ast^2 \mathbf{b}^2) |\abs{\im\left[\frac{h_T}{h}\right]}|^2+2\mathbf{b}^2\mathbf{a}^2a_\ast b_\ast J(z,h) }{J(z,h)}.\end{split}\]
Now we divide the proof into two cases
\subsubsection{Elastic case}
If $\mu>1$, then we take $a_\ast=\mathbf{b}/\mathbf{a}$ and have $$K[h]\ge \frac{(\mathbf{a}^2-b_\ast^2\mathbf{b}^2) \rho^2\abs{\im\left[\frac{h_T}{h}\right]}^2+2\mathbf{b}^2 b_\ast J(z,h)} {J(z,h)}$$
with the equality if $|h|_N=b_\ast \rho\abs{\im\left[\frac{h_T}{h}\right]}$.
Now we obtain
\[\begin{split}K[h]&\ge (\mathbf{a}^2-b_\ast^2\mathbf{b}^2)[2 B_\ast |h|_T-B_\ast^2 J(z,h)]+2 b_\ast\mathbf{b}^2\\&\ge  (\mathbf{a}^2-b_\ast^2\mathbf{b}^2)[2 B_\ast |h|\im \frac{h_T}{h}-B_\ast^2 J(z,h)]+2 b_\ast\mathbf{b}^2\end{split}\]
with the equality if \begin{equation}\label{eqs} h  \text{ is radial and } B_\ast = \frac{1}{|h_N|}. \end{equation}
Further for
\begin{equation}\label{tes}b_\ast= \frac{t\dot F(t) }{F(t)} \text{   and   }B_\ast=B_\ast(t,s)=\frac{F(t)}{s \dot F(t)}\end{equation}
we obtain $$2(\mathbf{a}^2-b_\ast^2\mathbf{b}^2) B_\ast |h|=X(t)\bydef\frac{2 \mathbf{a}^2 \left(-1+\mu^2\right)}{\sqrt{-1+\mu^2+t^2}}$$

$$-(\mathbf{a}^2-b_\ast^2\mathbf{b}^2)B_\ast^2 =Y(s)\bydef \frac{\mathbf{a}^2 \left(-1+\mu^2\right)}{s^2}.$$
Then we have $$\int_{\A} K[h]\ge \int_{\A} X(t) \im \frac{h_T}{h}dz+\int_{\A} Y(s)J(z,h)dz+\int_{\A}Z(t)dz=\int_{\A} K[h_\circ]$$ because $h_\circ$ satisfies \eqref{eqs}. Here $$Z(t)=b_\ast\mathbf{b}^2=\frac{t\dot F(t) }{F(t)}\mathbf{b}^2.$$

\subsubsection{Non-Elastic case}

If $\mu\le 1$, then we take $b_\ast=\mathbf{a}/\mathbf{b}$ and have
\[\begin{split}K[h]&\ge \frac{(\mathbf{b}^2-a_\ast^2\mathbf{a}^2) |h|_N^2+2\mathbf{a}^2 a_\ast J(z,h) }{J(z,h)}
\\&\ge (\mathbf{b}^2-a_\ast^2\mathbf{a}^2)[2 A_\ast |h|_N-A_\ast^2 J(z,h)]+2 a_\ast\mathbf{a}^2,\end{split}\]
with the equality if \begin{equation}\label{test}|h_T|=\mathbf{b} |h_N|\ \ \ \
\text{and    } A_\ast=\frac{1}{|h_T|}.\end{equation}

Now chose

$$a_\ast = \frac{F(t)}{t\dot F(t) }, \ \ \ A_\ast=A_\ast(t,s)=\frac{t}{s}.$$
Then $$2s(\mathbf{b}^2-a_\ast^2\mathbf{a}^2) A_\ast=U(t)\bydef\frac{2 \mathbf{b}^2  \left(1-\mu^2\right)}{ t}$$
and $$-(\mathbf{b}^2-a_\ast^2\mathbf{a}^2)A_\ast^2 =V(s)\bydef\frac{\mathbf{b}^2  \left(-1+\mu^2\right)}{s^2}$$
$$2 a_\ast\mathbf{a}^2=W(t)\bydef \frac{\mathbf{b}\mathbf{a} \sqrt{-1+\mu^2+t^2}}{t}.$$
Then
\[\begin{split}\int_{\A} K[h]&\ge \int_{\A} U(t) \frac{|h|_N}{|h|}{h}dz+\int_{\A} V(s)J(z,h)dz+\int_{\A}W(t)dz\\&=\int_{\A} K[h_\diamond],\end{split}\] because $h_\diamond$ satisfies \eqref{test}.
\section{Total energy}
For $\alpha, \beta>0$, $\mathbf{a},\mathbf{b}>0$, let $h\in \mathcal{F}(\A,\B)$, where $\A=A(1,r)$ and $\B=A(1,R)$ are circular annuli in the complex plane. Consider the total combined energy
\begin{equation}\mathfrak{E}[h]=\mathfrak{E}\left[\begin{array}{cc}
                                               \mathbf{a} & \mathbf{b} \\
                                               \alpha & \beta
                                             \end{array}
\right][h]\bydef \alpha\mathcal{E}[\mathbf{a} ,\mathbf{b}][h]+\beta\mathcal{E}[\mathbf{b},\mathbf{a}][h^{-1}].\end{equation}

Let $\mathbf{c}=\mathbf{a}/\mathbf{b}$.
Let $$\mathcal{F}_\circ(\A,\B)=\mathcal{F}(\A,\B)\bigcap \{h:\A\to \B, h(z)=H(t)e^{i(\theta+\alpha)}\}.$$

 \begin{theorem}\label{comtotalradial}
For every  $\mathbf{c}>0 $, the total combined energy integral $\mathfrak{E}: \mathcal{F}_\circ(\A,\B)\to \mathbf{R}$,  attains its minimum  for a radial diffeomorphism $h_\circ(z)=H(t)e^{i\theta}:\A\onto\B$, which satisfies one of the following three conditions listed below
$$\left\{
  \begin{array}{ll}
    H-t\mathbf{c} \dot H\equiv 0, & \hbox{for $t\in[1,r]$;} \\
    H-t\mathbf{c} \dot H>0, & \hbox{for $t\in[1,r]$;} \\
    H-t\mathbf{c} \dot H< 0, & \hbox{for $t\in[1,r]$.}
  \end{array}
\right.$$

\end{theorem}

In order to formulate the main result, and in connection with the previous theorem we consider the following triple of parameters $(\mathbf{c}, r, R)$ and say that they satisfy
$$\left\{
  \begin{array}{ll}
    \text{concavity case if} , & H-t\mathbf{c}^2 \dot H>0,  \hbox{for $t\in[1,r]$ and $R=H(r)$ and $\mathbf{c}\le 1$;} \\
   \text{convexity case if}, & H-t\mathbf{c}^2 \dot H<0,  \hbox{for $t\in[1,r]$ and $R=H(r)$ and $\mathbf{c}\ge 1$;}
  \end{array}
\right.$$

 Now we formulate the main result of this paper

 \begin{theorem}\label{comtotal}
Under concavity condition for $\mathbf{c}\le 1$ or convexity condition for $\mathbf{c}\ge 1$, the total combined energy integral $\mathfrak{E}: \mathcal{F}(\A,\B)\to \mathbf{R}$,  attains its minimum  for a radial mapping $h_\circ$. The same hold for the special case $R=r^{1/\mathbf{c}}$, for every $\mathbf{c}$. The minimizer is unique up to a rotation.
\end{theorem}

\begin{remark}
In \cite{arma} is considered the special case $\mathbf{c}=1$, and proved the result without restriction on convexity. However in this special case the resulting function $H$ is
$$\left\{
  \begin{array}{ll}
    concave , & \hbox{if $R>r$;} \\
    convex, & \hbox{if $R<r$;}\\
    identity, & \hbox{if $R=r$.}
  \end{array}
\right.$$
This in turn implies that our result covers the main result of Iwaniec and Onninen in \cite{arma}.
\end{remark}
\section{Radial total minimizers and the proof of Theorem~\ref{comtotalradial}}
Assume that $h$ is radial: $h(z)=H(t)e^{i\theta}$, $z=te^{i\theta}$.
Then $$\mathcal{E}[h]=2\pi\int_{1}^r\Lambda(t,H,\dot H)dt,$$ where $$\Lambda(t,H,\dot H)=\left(\alpha t+\beta\frac{t^2}{H\dot H}\right)\left(\mathbf{a}^2\dot H^2(t)+\mathbf{b}^2\frac{H^2(t)}{t^2}\right).$$ Now the Euler--Langrange equation $$\Lambda_H(t,H,\dot H)=\frac{\partial \Lambda_{\dot H}(t,H,\dot H)}{\partial t}$$ is equivalent with the equation  \begin{equation}\label{hsecond}H''(t)=\frac{H'(t)^2 \left(H(t)-\mathbf{c}^2 t H'(t)\right) \left(t+\gamma H(t) H'(t)\right)}{t H(t) \left(H(t)+\mathbf{c}^2 \gamma  t H'(t)^3\right)},\end{equation} where $$\gamma=\frac{\alpha}{\beta},  \ \ \ \mathbf{c}=\frac{\mathbf{a}}{\mathbf{b}}<1.$$
The equation \eqref{hsecond} can be written as
$$\left(H(t)-\mathbf{c} t H'(t)\right)'= \left(-H(t)+\mathbf{c} t H'(t)\right)M(t),$$ where $$M(t)=\frac{H'(t) \left(\mathbf{c}^2 t H'(t)+H(t) \left(-1+\mathbf{c}+\mathbf{c} \gamma  H'(t)^2\right)\right)}{H(t) \left(H(t)+\mathbf{c}^2 \gamma  t H'(t)^3\right)}.$$
Thus  $$ (H-t\mathbf{c}\dot H)'=- (H-t\mathbf{c}\dot H) M(t),$$ which is
equivalent to

$$\big[\log(H-t\mathbf{c}\dot H)\big]'=- M(t).$$
Thus $$\log(t\mathbf{c}\dot H-H) = \int_t^{t_1} (- M(\tau))d\tau+c,$$

\begin{equation}\label{lcon}t\mathbf{c}\dot H-H=c \exp[\int_{t_1}^{t} (- M(\tau))d\tau].\end{equation}

Now we prove that the diffeomorphic  solution of \eqref{hsecond} does
exist. The idea is simple, we want to reduce the equation
\eqref{hsecond} into
 an ODE of the first order, but to do this we assume that the diffeomorphic solution $H$ exists. This assumption is not harmful. Namely,
  the proof can be started from a certain first order ODE \begin{equation}\label{helpi}\Phi'=G[t,\Phi(t)],\end{equation} which has to do nothing
   with $H$ (see \eqref{berna} below). Then we solve \eqref{helpi} and, by using the solutions of it, we construct solutions of \eqref{hsecond}.
   Such a solution $H$ will be a diffeomorphism and so, it  will satisfy one of the three statements listed below. On the other hand if we have a diffeomorphic solution $H$ of \eqref{hsecond}, then it will satisfy the equation \eqref{lcon} for some continuous $M$ and this will imply the uniqueness of solution $H$.

So if $H$ is a strictly increasing $C^2$ diffeomorphism defined in a domain $(a,b)$ that solves the equation \eqref{hsecond},
then $\dot H(s)>0$ and from \eqref{lcon}, we conclude that there are three possible cases:
\begin{itemize}
\item {\bf Case~1} $c=0$. Then $H-t\mathbf{c} \dot H\equiv 0$, or what is the same $H(t)=t^{1/\mathbf{c}}$, and this produces the mapping $h(z)=|z|^{1/\mathbf{c}}e^{i\theta}$, so in this case $${R}=r^{1/\mathbf{c}}.$$
\item {\bf Case~2} $c>0$. Then $t\mathbf{c}\dot H-H>0$ and thus $s(t)=\frac{H(t)}{t^{1/\mathbf{c}}}$ is monotone increasing. Then $${R}>{r^{1/\mathbf{c}}}.$$
\item {\bf Case~3} $c<0$. Then $t\mathbf{c}\dot H-H<0$ and thus $s(t)=\frac{H(t)}{t^{1/\mathbf{c}}}$ is  monotone decreasing. Then $$R<r^{1/\mathbf{c}}.$$
\end{itemize}
Now if $c>0$ and $\mathbf{c}<1$, then we define elasticity function $\eta$
$$\eta(t)=\frac{H(t)}{t}$$ and obtain \begin{equation}\label{etapoz}\eta'(t)=\frac{t\dot H(t)-H(t)}{t^2}\ge \frac{\mathbf{c} t\dot H(t)-H(t)}{t^2}> 0.\end{equation}
We  take the new variable $s=\frac{H(t)}{t}$ and the new function
\begin{equation}\label{prin}\Phi(s)=\frac{1}{s}\dot H\left(\frac{H(t(s))}{t(s)}\right),\end{equation} where $t(s)$ is the inverse of $s=s(s)$. Then we obtain $$\ddot H(t)=\frac{\dot \Phi(s) (t \dot H(t) - H(t))}{t^2}.$$
We can without loss of generality assume that $\mathbf{c}<1$, otherwise $\frac{1}{\mathbf{c}}<1$ and consider the duality problem which is the same but instead of $\mathbf{c}$ has $1/\mathbf{c}$.

The auxiliary equation which we have to solve is

\begin{equation}\label{berna}\Phi'(s)=-\frac{\Phi(s) \left(\mathbf{c}^2 \Phi (s)^2-1\right) \left(1+s^2 \gamma \Phi (s)^2\right)}{s (\Phi(s)-1) \left(1+\mathbf{c}^2\gamma  s^2 \Phi^3(s)\right)}, \ \ \ s\in(0,\infty),\end{equation} i.e.

$$\Phi'(s)= U(s,\Phi(s))$$ where $$U(s,y)=-\frac{y \left(\mathbf{c}^2 y^2-1\right) \left(1+s^2 \gamma y^2\right)}{s (y-1) \left(1+\mathbf{c}^2\gamma  s^2 y^3\right)}.$$

Let $\Phi_p=\Phi$ be a solution so that
$$\Phi(1)=q,\ \ \  0<q<1 \vee q>1.$$


\begin{lemma}\label{boundf}
For fixed $0<\mathbf{c}<1$ and
\begin{enumerate}
  \item $q\in (0,1)$, there is $0<a(q)<1$ such that  the solution $ \Phi_q$ is a decreasing diffeomorphism of its maximal interval $(a(q),\infty)$ onto $\left(0,1\right)$ (see Figure~1).

 \item  $q\in (1, \frac{1}{\mathbf{c}})$ the solution $ \Phi_q$ is an increasing diffeomorphism of $(1,\infty)$ onto $\left(q,\frac{1}{\mathbf{c}}\right)$ (see Figure~2).
  \item if $q=1/\mathbf{c}$, the solution is $\Phi_q(s)\equiv q$.

  \item $q\in(\frac{1}{\mathbf{c}},+\infty)$ the solution $ \Phi_q$ is a decreasing diffeomorphism of $(0,\infty)$ onto $\left(\frac{1}{\mathbf{c}},
  +\infty\right)$ (see Figure~3).

\end{enumerate}

\end{lemma}
\begin{figure}[htp]
\centering
\includegraphics{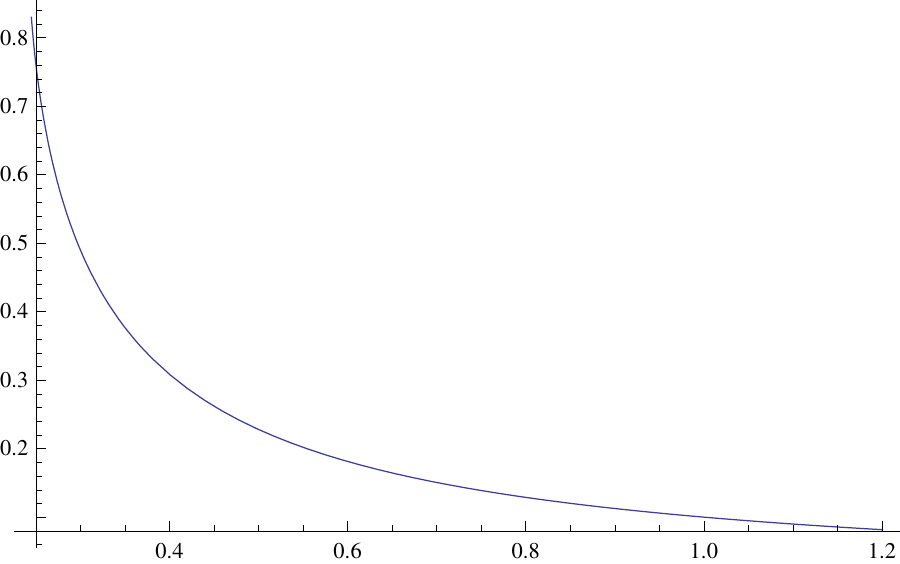}
\caption{The graph of $\Phi_q$ for $q=1/10$, $\mathbf{c}=1/2$, $\gamma=1$}
\end{figure}

\begin{figure}[htp]
\centering
\includegraphics{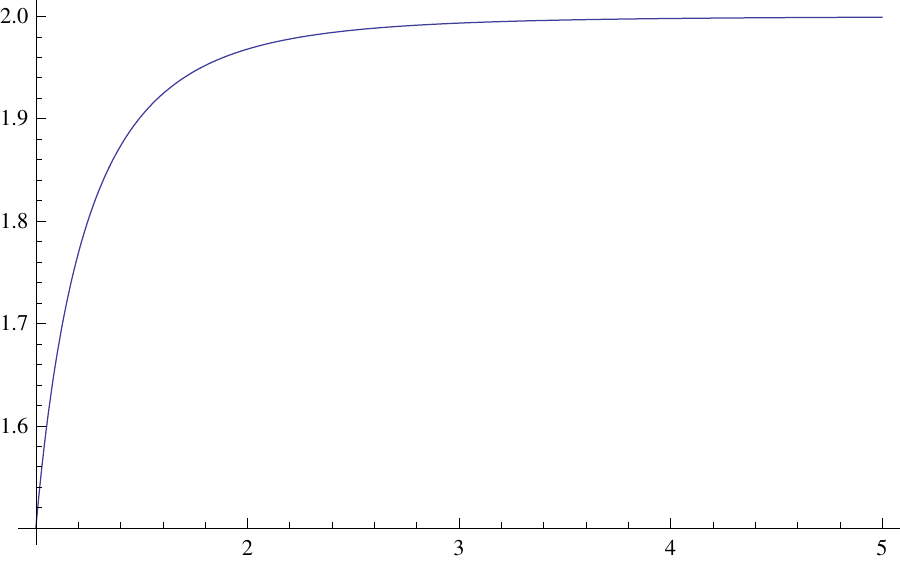}
\caption{The graph of $\Phi_q$ for $q=1/3$, $\mathbf{c}=1/2$, $\gamma=1$}
\end{figure}

\begin{figure}[htp]\label{f1}
\centering
\includegraphics{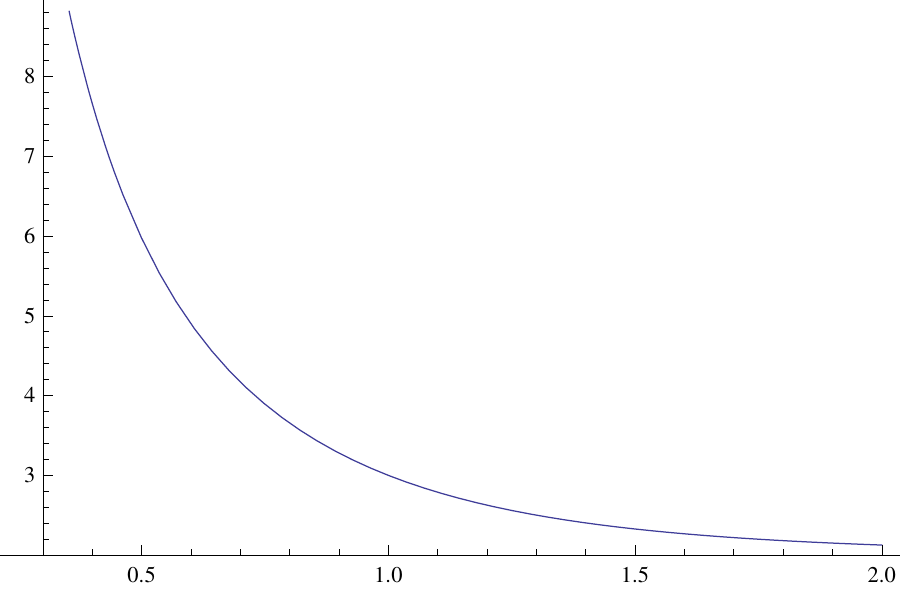}
\caption{The graph of $\Phi_q$ for $q=3$, $\mathbf{c}=1/2$, $\gamma=1$}
\end{figure}
\begin{proof}
(1) Let $q\in(0,1)$ and  let be  $(a(q),\omega)$ the maximal interval of $\Phi_q$ containing $1$. Then $\Phi_q(a(q))=1$, otherwise we could continue $\Phi_q$ below $a(q)$.   From \eqref{berna} by integrating we conclude that $$\Phi(1)-1-\log \Phi(1)=\int_{a(q)}^1 \frac{ \left(1-\mathbf{c}^2 \Phi (s)^2\right) \left(1+s^2\gamma  \Phi (s)^2\right)}{s  \left(1+\mathbf{c}^2\gamma  s^2 \Phi^3(s)\right)}ds.$$ Hence $$\lim_{q\to 0}a(q)=0.$$ If the maximal interval is $(a(q),\omega)$, where $\Phi_p$ decreases, then $\Phi_p(\omega)=0$, but then it coincides with the constant solution $\Phi_0\equiv 0$, and this is impossible.

 So  $\Phi_q$ has as its maximal interval $(a(q),\infty)$. Also $\lim_{s\to \infty} \Phi_p(s)=0$, otherwise in view of \eqref{berna},
we would have $$\int_1^M\frac{\Phi_p'(s)}{\Phi_p(s)}ds=\log \frac{\Phi_p(M)}{\Phi_p(1)}\ge C \log M$$ which is impossible.

(2) If $1<q<1/\mathbf{c}$, then $\Phi_q'(1)>0$. The rest of the proof is similar to the previous one. It should be noticed that $\Phi_q$ does not reaches $1/\mathbf{c}$, otherwise it will be stationary as in the previous case.

(4) Let $q\in(\frac{1}{\mathbf{c}},+\infty)$. Then $\Phi_q'(1)<0$, and so $\Phi_q$ is decreasing  in the maximal interval $(\varpi,\omega)$ containing $1$. We state that $\omega=\infty$ and $\varpi=0$. 
{Further, if  $\omega<\infty$, then $\Phi_q(\omega)=1/\mathbf{c}$ and $\Phi_q'(\omega)=0$. But then $\Phi_q$ would coincide with the constant solution $\Phi_\mathbf{c}(s)\equiv {1/\mathbf{c}}$. I.e. there is a interval around $\omega$ so that $\Phi_\mathbf{c}$ coincide. This is impossible since $\Phi_q$ is strictly decreasing in $(\varpi,\omega)$. Thus $\omega=\infty$. Further $\lim_{s\to\infty} \Phi_q(s)=\frac{1}{\mathbf{c}}$. If not, then since $\Phi_q$ is decreasing we would have}
$\Phi_q(s)>1/\mathbf{c}+\epsilon$, $\epsilon>0$ for all $s$. Then we would have $$\left|\frac{\Phi_q'(s)}{\Phi_q(s)}\right|\ge C\frac{1}{s}.$$ By integrating in $[1,s]$ this implies that
$$\log \Phi_q(s)/\Phi_q(1)\ge C \log s,$$ so $\Phi$ tends to infinity when $s$ tends to infinity, which is a contradiction.  Similarly, if $\varpi$ is $>0$, then $$\log \Phi_q(s)/\Phi_q(1)\le C \log \varpi$$ and thus there is a limit $$\Phi_q(\varpi):=\lim_{t\to\varpi} \Phi_q(s)\le \Phi_q(1)e^{C \log \varpi}.$$ So $(\varphi, \Phi_q(\varpi))$ is the point of continuity of $(x,y)\to U(x,y)$, and thus $\Phi_q$ can be continued below $\varpi$. This implies that $\varpi=0$. Moreover we have $\lim_{s\to 0} \Phi_q(s)=+\infty$, which can be proved in a similar fashion.

\end{proof}

\begin{lemma}\label{interes}
For fixed $s>0$, $$\lim_{q\uparrow \infty } \Phi_q(s)=+\infty \ \  \ \text{and} \ \  \ \lim_{q\downarrow \frac{1}{\mathbf{c}} } \Phi_q(s)=\frac{1}{\mathbf{c}}.$$
\end{lemma}
\begin{proof}
We use Lemma~\ref{boundf}. If $s<1$, then $\Phi_q(s)\ge \Phi_q(1)=q$ and so $\lim_{q\uparrow \infty} \Phi_q(s)=\infty$. Since two integral curves never intersect, it follows that for fixed $s$, $q\to \Phi_q(s)$ is increasing. Assume that $s>1$ and $ \Phi_{q}(s)< M<\infty$ for every $q$. Then there exist a  solution $\Phi_0$ with $\Phi_0(s)=M$ and with $(0,\infty)$ as its maximal interval. But then, $\Phi_0(1)=q_0$ which is a contradiction with the fact that $\Phi_{q_0}(s)<M$. Thus $\lim_{q\uparrow \infty } \Phi_q(s)=+\infty$
as claimed.

Further if $s>1$, then $q=\Phi_q(1)> \Phi_q(s)\ge 1/\mathbf{c}$ and so $\lim_{q\downarrow \frac{1}{\mathbf{c}}} \Phi_q(s)=\frac{1}{\mathbf{c}}$. If for some $s<1$, $\lim_{q\downarrow \frac{1}{\mathbf{c}}}\Phi_q(s)>1/\mathbf{c}+\epsilon$, where $\epsilon>0$, then as before this leads to contradiction.

\end{proof}

Similarly by using Lemma~\ref{boundf}~2), we prove the following lemma:

\begin{lemma}\label{interes2}
For fixed $s>1$, $$\lim_{q\uparrow \frac{1}{\mathbf{c}}} \Phi_q(s)=\frac{1}{\mathbf{c}} \ \  \ \text{and} \ \  \ \lim_{q\downarrow 1 } \Phi_q(s)=1.$$

\end{lemma}

By using Lemma~\ref{boundf}~1) we get
\begin{lemma}\label{interes1}
For fixed $s>1$, $$\lim_{q\uparrow 1} \Phi_q(s)=1 \ \  \ \text{and} \ \  \ \lim_{q\downarrow 0 } \Phi_q(s)=0.$$

\end{lemma}
The following lemma is just a reformulation of Theorem~\ref{comtotalradial}.
\begin{lemma}\label{hsol}
Let $0<\mathbf{c}<1$ and $q\in(0,1)\cup(1,\infty)$. Then there exist a solution $H=H_q(t)$ of the ODE equation
\begin{equation}\label{hsecond1}
\left\{
  \begin{array}{ll}
    H''(t)=\frac{H'(t)^2 \left(H(t)-\mathbf{c}^2 t H'(t)\right) \left(t+\gamma  H(t) H'(t)\right)}{t H(t) \left(H(t)+\mathbf{c}^2 \gamma  t H'(t)^3\right)} & \hbox{ } \\
    H_q(1)=1  \hbox{and } H'_q(1)=q.
  \end{array}
\right.
 \end{equation} Then $H_q$ has $[1,\infty)$ as its maximal interval, and $H_q$ is an increasing diffeomorphism of $[1,\infty)$ onto itself.  Moreover, for fixed $r>1$ and $R> 1$ there exists $q=q(r,R,\mathbf{c})$ so that $H_q[1,r]=[1,R]$. The function $q$ is continuous in its parameters.
\end{lemma}
\begin{proof}[Proof of Lemma~\ref{hsol}] Let $\Phi_q$ be the solution of \eqref{berna} obtained by Lemma~\ref{boundf}, and observe the boundary value problem
\begin{equation}\label{firstq}\left\{
  \begin{array}{ll}
    \dot H(t)=\frac{H(t)}{t}\Phi_q\left(\frac{H(t)}{t}\right), & \hbox{}\\
    H(1)=1. &
  \end{array}
\right.\end{equation}

Using the Picard-Lindel\"of theorem, we observe that for fixed
$q$ there is  exactly one smooth solution
$H=H_{q}$ of the problem \eqref{firstq}.
Let $[1,\omega)$ be the maximal interval of $H_q$. If $\omega<\infty$ and $q\in(0,1)\cup (1,1/\mathbf{c})$, then by Lemma~\ref{boundf}, we have $$\frac{H_q'(t)}{H_q(t)}\le \frac{C}{t},$$ and by integrating in $[1,\omega)$ we obtain $$\log(H_q(\omega-0))\le C \log\omega.$$ So $(\omega,H_q(\omega))$ is the point of continuity of the function $$(x,y)\mapsto \frac{y}{x}\Phi_q\left(\frac{y}{x}\right).$$ Thus we can continue the solution above $\omega$, which is a contradiction. Thus $\omega=\infty$ if $q\in(0,1)\cup (1,1/\mathbf{c})$. If $q>1/\mathbf{c}$, and $[1,\omega)$ is the maximal interval of $H_q$ with $\omega<\infty$, then if $H_q(\omega)< \infty,$ then $H_q$ can be continued above $\omega$. If $H_q(\omega)=\infty$, by Lemma~\ref{boundf}, $$\lim_{s\to\omega}\Phi_q\left(\frac{H_q(s)}{s}\right)=\frac{1}{\mathbf{c}}.$$ This again leads to a contradiction.
Moreover, if $1\le t\le r$, then from Lemma~\ref{boundf}, 1)  $$(\log H_q(t))'=\frac{1}{t}\Phi_q\left(\frac{H_q(t)}{t}\right)$$ tends uniformly  to zero when $q\downarrow 0$. By integrating in $[1,r]$ we conclude that $\log H_q(t)$ tends to zero for every fixed  $1\le t\le r$. So $H_q(r)$ tends to $1$ as $q\downarrow 0$. Further $$\lim_{q \uparrow 1} H_q(r).$$
Similarly we show that $H_q(r)$ tends to $r^{1/\mathbf{c}}$ when $q \uparrow \frac{1}{\mathbf{c}}$.

From Lemma~\ref{interes}, we obtain that for every $R> r^{1/\mathbf{c}}$ there is $q$ so that $H_q[1,r]=[1,R]$.

\end{proof}
\begin{proposition}
Let $0<\mathbf{c}<1$ and assume that $H=H_q:[1,r]\to [1,R]$ is a diffeomorphic solution of \eqref{hsecond} with $0<q<1$. Then $H$ is a convex function.
\end{proposition}
\begin{proof}
In view of \eqref{lcon}, we have $t\mathbf{c}\dot H-H=c=\mathbf{c} q-1<0. $ It follows that $$t\mathbf{c}^2\dot H-H\le \mathbf{c} q-1<0.$$ According to \eqref{hsecond}, $H''>0$ which implies that $H$ is convex.
\end{proof}
\begin{remark}
If $\mathbf{c}< 1$, then the identity  is  not a minimizer of total combined functional. It follows that the radial solution of the equation \eqref{hsecond} that maps the interval $[1,r]$ onto itself is convex. The graphic of such a solution is given in the following figure.
\begin{figure}[htp]\label{poi}
\centering
\includegraphics{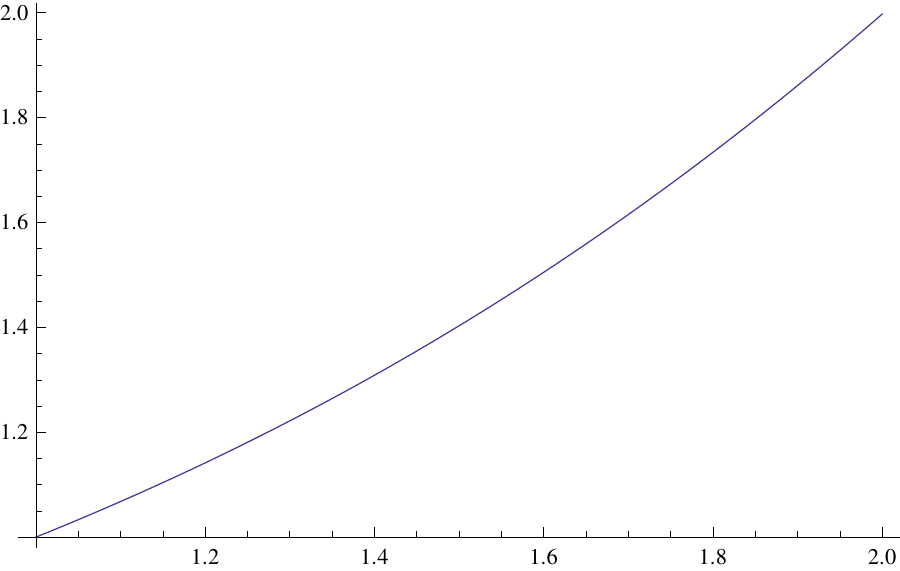}
\caption{The solution $H_q$ that maps the interval $[1,2]$ onto itself for $\mathbf{c}=1/2$ and $\gamma=1$ with $q<1$.}
\end{figure}
\end{remark}

\begin{proposition}\label{propo}
a) For $r>1$ and $R> r$, there is $\mathbf{c}_\circ=\mathbf{c}_\circ(r,R)<1$ such that the diffeomorphic solution $H=H_q:[1,r]\to [1,R]$ of \eqref{hsecond} obtained in Lemma~\ref{hsol} satisfies the {\bf concavity} condition \begin{equation}\label{gamasquare}\mathbf{c}^2t \dot H(t)\ge H(t), \ \ \ 1\le t \le r\end{equation} for every $ \mathbf{c}_\circ\le \mathbf{c}\le 1$. In other words, in view of \eqref{hsecond} the solution is concave.

b) In the opposite direction assume that $\mathbf{c}<1$ and $q>1/\mathbf{c}^2$.  If the condition \eqref{gamasquare} is satisfied then  $$r\le \frac{\log (q \mathbf{c}^2)}{\log (1/\mathbf{c}^2-1)}$$ and
$$r^{1/\mathbf{c}^2}\le R \le r q \mathbf{c}^2.$$
Moreover if for some fixed $\mathbf{c}<1$ the concavity condition is satisfied in $[1,r]$, then we have $$r< \frac{1}{1-\mathbf{c}^2+\mathbf{c}^2/R}.$$
\end{proposition}

\begin{figure}[htp]
\centering
\includegraphics{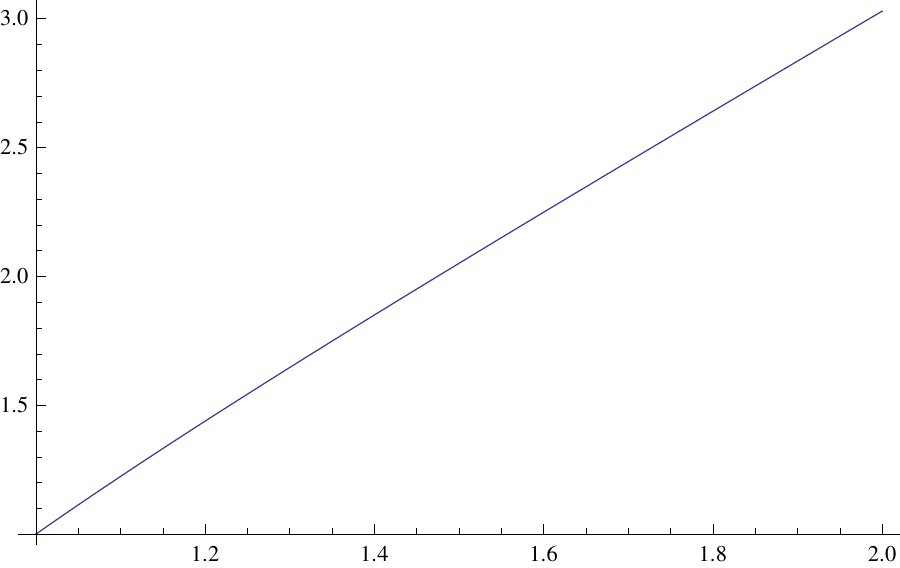}
\caption{The concave graphic of the solution $H$ for, $\mathbf{c}=9/10$, $\gamma=1$ that maps $[1,2]$ onto $[1,3]$}
\end{figure}

\begin{proof}

For fixed $r>1$ and $R>r$, there is $q=q(r,R,\mathbf{c})$ so that $H_{q,\mathbf{c}}[1,r]=[1,R]$ is the diffeomorphic solution of \eqref{hsecond}. Moreover   $$H{q,\mathbf{c}}(t)-\mathbf{c}^2 t H_{q,\mathbf{c}}'(t)\to H_{q,1}(t)- t H_{q,1}'(t)$$ uniformly on $[1,r]$ as $\mathbf{c}\uparrow 1$. From \eqref{lcon} we have $$H_{q,1}(t)- t H_{q,1}'(t)\equiv c ,  \ \ t\in [1,r]. $$ Now since $R>r$, the function  $P(t)=\frac{H(t)}{t}$ satisfies the condition $1=P(1)<P(r)=R/r$, and therefore
$$0<\partial_t \frac{H_{q,1}(t)}{t}=-\frac{(H_{q,1}(t)- t H_{q,1}'(t))}{t^2}=\frac{-c}{t^2}.$$ Thus $c<0$.  Thus there is $\mathbf{c}<1$ so that $$H{q,\mathbf{c}}(t)-\mathbf{c}^2 t H_{q,\mathbf{c}}'(t)<0, \ \ \ 1\le t \le r.$$
Further by integrating the inequality $$\frac{H'(t)}{H(t)}\ge \frac{1}{\mathbf{c}^2 t}$$ in $[1,r]$ we obtain $$R=H(r)>r^{1/\mathbf{c}^2}.$$
 Further let $V(s)\bydef s\Phi_q(s),$  where $\Phi_q$ is the solution of \eqref{berna}. Then \begin{equation}\label{vp}V'(s)= \frac{V(s)^2 \left(-s+\mathbf{c}^2 V(s)\right) (1+s V(s))}{s (s-V(s)) \left(s+\mathbf{c}^2 V(s)^3\right)},\end{equation} and $$H'(t)=V\left(\frac{H(t)}{t}\right) .$$ It follows that $$H''(t)=V'\left(\frac{H(t)}{t}\right)\frac{tH'- H}{t^2}.$$ Now if $V=V_q$ is a solution of \eqref{vp} with $V(1)=q>1$, then $H'(1)=q>1$. Further if $t_\circ>1$ is the smallest zero of  $H''(t_\circ)$ then $s_\circ=H(t)/t$ is the smallest zero of $V'(s)$. Then $$V(s_\circ)=\frac{s_\circ}{\mathbf{c}^2}.$$ Further $V'(1)<0$, and thus $V$ is decreasing in $[1,s_\circ]$. Thus $$ q=V(1)> V(s)=\frac{s_\circ}{\mathbf{c}^2}.$$
So $s_\circ< q \mathbf{c}^2.$ Therefore $H(t_\circ)< t_\circ q \mathbf{c}^2.$

 From
$$R-1=H(r)-H(1)=\int_1^r V\left(\frac{H(t)}{t}\right)dt,$$ and the assumption that $V$ is decreasing and $\frac{H(t)}{t}$ is increasing we obtain that $$\frac{R-1}{r-1}=\frac{H(r)-H(1)}{r-1}\ge V \left(\frac{H(r)}{r}\right)> \frac{H(r)}{r}\frac{1}{\mathbf{c}^2}=\frac{R}{r}\frac{1}{\mathbf{c}^2}.$$ Thus $$\mathbf{c}^2> \frac{R}{r}\frac{r-1}{R-1}$$ and hence $r< \frac{1}{1-\mathbf{c}^2+\mathbf{c}^2/R}.$ This finishes the proof.
\end{proof}
 \section{Proof of Theorem~\ref{comtotal}} We have to estimate the total combined energy $$\mathfrak{E}[h]= \alpha\mathcal{E}[\mathbf{a},\mathbf{b}][h]+\beta\mathcal{E}[\mathbf{b},\mathbf{a}][h^{-1}]$$ for $h\in F(\A,\B)$. Here we have $0<\mathbf{c}<1$.

We make use of the following simple inequalities
\begin{equation}\label{general}\begin{split}\mathbf{a}^2 |h_N|^2+\mathbf{b}^2|h_T|^2&\ge (\mathbf{a}^2-\mathbf{b}^2 a^2)|h_N|^2 +(\mathbf{b}^2-\mathbf{a}^2 b^2) |h_T|^2\\& + 2 a b \mathbf{a}\mathbf{b}  |h_T| |h_N|\end{split}\end{equation}
\begin{equation}\label{general1}\begin{split}\frac{\mathbf{a}^2 |h_N|^2+\mathbf{b}^2|h_T|^2}{J(h,z)}
&\ge \frac{(\mathbf{a}^2-\mathbf{b}^2 a_\ast^2)|h_N|^2 +(\mathbf{b}^2-\mathbf{a}^2 b_\ast^2) |h_T|^2 }{J(h,z)}\\&+\frac{ 2 a_\ast b_\ast \mathbf{a}\mathbf{b}  |h_T| |h_N|}{J(h,z)}.\end{split}\end{equation}
The equality is attained in \eqref{general}  if and only if
\begin{equation}\label{mamire}
\mathbf{b}a|h_N|=\mathbf{a}b |h_T|
\end{equation}
and in \eqref{general1} if and only if
\begin{equation}\label{mamire}
\mathbf{b}a_\ast|h_N|=\mathbf{a}b_\ast |h_T|.
\end{equation}
Then
\begin{equation}\label{split}\begin{split}
\mathfrak{D}[h]&\bydef \alpha(\mathbf{a}^2|h_N|^2 +\mathbf{b}^2|h_T|^2)+\beta\frac{\mathbf{a}^2|h_N|^2 +\mathbf{b}^2|h_T|^2}{J(z,h)}\\
&\ge \alpha(\mathbf{a}^2-\mathbf{b}^2 a^2)(2|h|_N-A)A
\\&+\alpha(\mathbf{b}^2-\mathbf{a}^2 b^2)(2|h|\Im \frac{h_T}{h}-B)B
\\&+ 2\mathbf{a}\mathbf{b} \alpha ab J(z,h)
\\&+ \beta(\mathbf{a}^2-\mathbf{b}^2a_\ast^2)(2|h|_N-A_\ast J(z,h))A_*
\\&+\beta(\mathbf{b}^2-\mathbf{a}^2 b_\ast^2)(2|h|\Im \frac{h_T}{h}-B_\ast J(z,h))B_\ast
\\&+2\mathbf{a}\mathbf{b} na_\ast b_\ast.
\end{split}
\end{equation}

The equality holds at a given point $z$ if and only if

\begin{equation}\label{247}\mathbf{b}a |h_N | = \mathbf{a} b |h_T |, \ \ \  \mathbf{b} a_* |h_N | = \mathbf{a}b_* |h_T | \end{equation}

\begin{equation}\label{248}A = |h_N |\text{ except for } a = \frac{\mathbf{a}}{\mathbf{b}},  A_* = \frac{1}{|h_T |} \text{ except for } a_* = \frac{\mathbf{a}}{\mathbf{b}}\end{equation}

\begin{equation}\label{249}B = |h_T |\text{ except for }b  = \frac{\mathbf{b}}{\mathbf{a}}, B_* = \frac{1}{|h_N |}\text{ except for } b_*  = \frac{\mathbf{b}}{\mathbf{a}} \end{equation}

\begin{equation}\label{250}h \overline{h_N}\in  \R_+ \text{ and }  \overline{ h} h_T \in  i
\R_+ \end{equation}
\subsection{Proof of balanced case: $R=r^{1/\mathbf{c}}$}
By taking $a=a_\ast=\mathbf{c}$ and $b=b_\ast=1/\mathbf{c}$, the inequality \eqref{split} reduces to the following simple inequality

\[\begin{split}
\mathfrak{D}[h]&=\alpha(\mathbf{a}^2|h_N|^2 +\mathbf{b}^2|h_T|^2)+\beta\frac{\mathbf{a}^2|h_N|^2 +\mathbf{b}^2|h_T|^2}{J(z,h)}
\\&\ge 2\mathbf{a}\mathbf{b} \alpha ab J(z,h)+2\mathbf{a}\mathbf{b} \beta a_\ast b_\ast\\&=
2\mathbf{a}\mathbf{b} \alpha  J(z,h)+2\mathbf{a}\mathbf{b} \beta
\end{split}
\]
because  $$ab=a_*b_*=\frac{\mathbf{a}}{\mathbf{b}}\frac{\mathbf{b}}{\mathbf{a}}=1.$$
So \[\begin{split}\int_{\A}\mathfrak{D}[h]dz&\ge 2\mathbf{a}\mathbf{b}  \alpha \int_{\A}  J(z,h)dz+2\mathbf{a}\mathbf{b}  \beta \int_{\A}  dz
\\&=2\mathbf{a}\mathbf{b}ab \alpha\pi (R^2-1) +2\mathbf{a}\mathbf{b} \beta \pi (r^2-1)\\&=2\mathbf{a}\mathbf{b} \alpha\pi (r^{2/\mathbf{c}}-1)+ 2\mathbf{a}\mathbf{b} \beta \pi (r^2-1)\\&=\int_{\A}\mathfrak{D}[h_\circ]dz, \end{split}\] where $h_\circ (z)= |z|^{1/\mathbf{c}-1} z$. Namely in this special case
$$\mathbf{a}^2|h_N|^2 +\mathbf{b}^2|h_T|^2=2\mathbf{a}\mathbf{b} |h_N|\cdot |h_T|.$$

\begin{figure}[htp]\label{poi}
\centering
\includegraphics{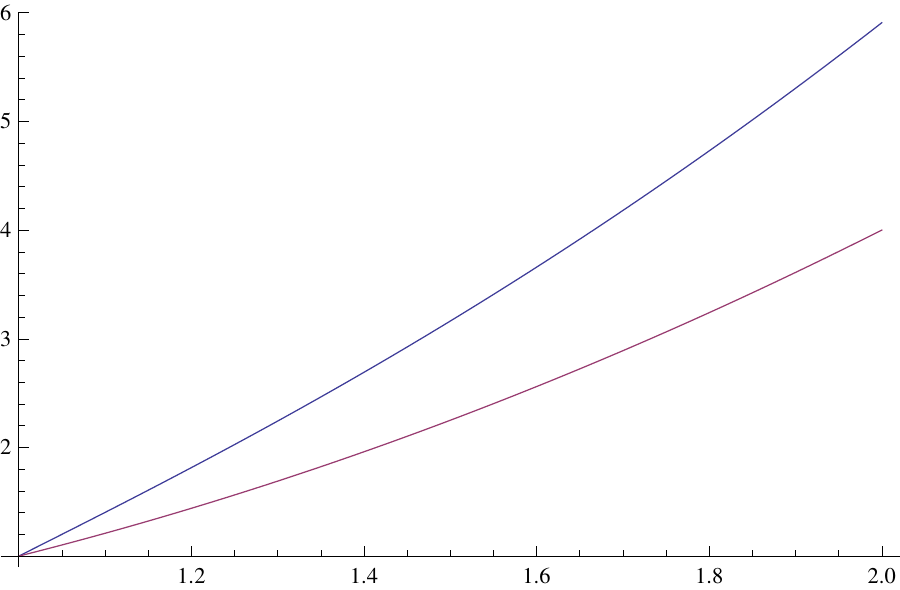}
\caption{These two curves are graphs of the solutions with the initial conditions $H'(1)=1/\mathbf{c}$ and $H'(1)=1/\mathbf{c}^2$ respectively, and $1<t<2$.}
\end{figure}

\subsection{Proof of concavity case: $H_q''(t)\le 0$ ($\mathbf{c}<1$).}
Take
\begin{equation}\label{bbsg}b=\frac{\mathbf{b}}{\mathbf{a}}=b_\ast=\frac{1}{\mathbf{c}}>1.\end{equation}

Thus

\[\begin{split}
\mathfrak{D}[h]
&\ge \alpha(\mathbf{a}^2-\mathbf{b}^2 a^2)(2|h|_N-A)A
\\&+ 2\mathbf{a}\mathbf{b} \alpha ab J(z,h)
\\&+ \beta(\mathbf{a}^2-\mathbf{b}^2a_\ast^2)(2|h|_N-A_\ast J(z,h))A_*
\\&+2\mathbf{a}\mathbf{b} \beta a_\ast b_\ast.
\end{split}
\]

So
\[\begin{split}
\mathfrak{D}[h]
&\ge 2\left[(\mathbf{a}^2-\mathbf{b}^2 a^2)Am+(\mathbf{a}^2-\mathbf{b}^2a_\ast^2)A_\ast \beta\right]|h|_N
\\&+ \left[2\mathbf{b}^2 a \alpha-(\mathbf{a}^2-\mathbf{b}^2a_\ast^2)A^2_\ast \beta\right] J(z,h)
\\&+2\mathbf{b}^2 a_\ast \alpha -  (\mathbf{a}^2-\mathbf{b}^2 a^2)A^2 \beta.
\end{split}
\]

Equality hold at a given point if \begin{equation}\label{aas}A=|h|_N, \ \ \ \ A_\ast=\frac{1}{|h_T|}\end{equation}
and
\begin{equation}\label{aht} a|h_N|=  |h_T| \ \ \ \  a_\ast |h_N|=|h_T|.\end{equation}
Chose now $$a=a(s)=\frac{s\dot F(s)}{F(s)},$$ $$a_\ast=a_\ast(t)=\frac{H(t)}{t \dot H(t)},$$
 $$A=A(t,s)=\dot H(t)\sqrt{\frac{\mathbf{a}^2-\mathbf{b}^2a_\ast^2}{\mathbf{a}^2-\mathbf{b}^2a^2}}$$
and
  $$A_\ast=A_\ast(t,s)=\frac{F(s)}{s}\sqrt{\frac{\mathbf{a}^2-\mathbf{b}^2a^2}{\mathbf{a}^2-\mathbf{b}^2a_\ast^2}}.$$
Then $$\left[2\mathbf{b}^2 am-(\mathbf{a}^2-\mathbf{b}^2a_\ast^2)A^2_\ast \beta\right] J(z,h)=U(s) J(z,h)$$
and
$$2\mathbf{b}^2 a_\ast \beta -  (\mathbf{a}^2-\mathbf{b}^2 a^2)A^2\alpha=V(t)$$ so they are free Lagrangians.

So \begin{equation}\label{dell} \mathfrak{D}[h]\ge U(s) J(z,h)+V(t)+X\end{equation}
where $$X=2\left[(\mathbf{a}^2-\mathbf{b}^2 a^2)Am+(\mathbf{a}^2-\mathbf{b}^2a_\ast^2)A_\ast \beta\right]|h|_N.$$ It remains to estimate $X$ by means of free Lagranians. First of all $X$ is equal to
 $$2 \sqrt{\mathbf{a}^2-\mathbf{b}^2 a^2} \sqrt{\mathbf{a}^2-\mathbf{b}^2 a_\ast^2}\left(\alpha\dot H(t)+\beta\frac{F(s)}{s}\right)|h|_N.$$
Let $$\Gamma=\Gamma(t,s)=2t \sqrt{\mathbf{a}^2-\mathbf{b}^2 a^2} \sqrt{\mathbf{a}^2-\mathbf{b}^2 a_\ast^2}\left(\alpha\dot H(t)+\beta\frac{F(s)}{s}\right).$$
Let $$\mathcal{A}(t,s)=\int_1^s\Gamma(t,\tau)d\tau -\int_1^t \int_1^{H(\varsigma)}\Gamma_{\varsigma}(\varsigma,\tau)d\tau d\varsigma.$$
Hence $$\mathcal{A}_t(t,s)=\int_{H(t)}^s\Gamma_t(t,\tau)d\tau.$$
The aim is of to obtain the following lemma
\begin{lemma}\label{71}
Let $h\in \mathcal{F}_\circ(\mathbb{A}, \B)$. Then
\begin{equation}\label{gamast}\Gamma(t,s)\frac{|h|_N}{|z|}\ge \frac{\mathcal{A}_t(t,s)+\mathcal{A}_s(t,s)|h|_N}{|z|}\end{equation} where $s=|h(z)|$ and $t=|z|$. The equality is attained in \eqref{gamast} if and only if $h(z)=H(|z|)$, where $H$ satisfies the equation \eqref{hsecond}.
\end{lemma}
\begin{proof}
Observe first that $$\mathcal{A}_s(t,s)=\Gamma(t,s).$$ So we need only to show that $\mathcal{A}_t(t,s)\le 0$. First we have $\mathcal{A}_t(t, H(t))=0$. Further we have $$\partial_s\mathcal{A}_t(t, s)=\Gamma_t(t,s),$$
 having in mind the equation \eqref{hsecond} we obtain after straightforward but lengthly computations
$$\Gamma_t =\frac{2 \alpha \beta (\frac{F(s)}{s} -\frac{ t}{H(t)}) \sqrt{\mathbf{a}^2-\mathbf{b}^2 a^2} \sqrt{\mathbf{a}^2-\mathbf{b}^2 b^2} H'(t)^2 \left(\mathbf{a}^2 t H'(t)-\mathbf{b}^2 H(t)\right)}{ \left(\beta \mathbf{b}^2 H(t)+\alpha \mathbf{a}^2 t H'(t)^3\right)}.$$
Since $\frac{F(s)}{s}$ is decreasing because $$\frac{d}{ds}\frac{F(s)}{s}=F(s)\left(\frac{s F'(s)}{F(s)}-1\right)\le F(s)(\mathbf{c}-1)<0,$$ it follows that $$\frac{F(s)}{s}-\frac{ t}{H(t)}$$ is decreasing and is vanishing for $s=H(s)$. Combining with the fact that in concavity case we have
$$\frac{ t H'(t)}{H}\ge \frac{1}{\mathbf{c}^2}, \ \  1\le t \le r, $$ we conclude that
$$\Gamma_t \left\{
                   \begin{array}{ll}
                     >0, & \hbox{$1\le s<H(t)$ } \\
                     =0, & \hbox{$s=H(t)$} \\
                     <0, & \hbox{$H(t)<s<R$.}
                   \end{array}
                 \right.$$

Thus $\mathcal{A}_t(t, s)\le \mathcal{A}_t (t, H(t))=0$.

\end{proof}

Now we finish the proof of concavity case.

We have
\[\begin{split}
\int_{\A}\mathfrak{D}[h]dz&\ge\int_{\A}\left[  U(s) J(z,h)+V(t)+X\right] dz
\\& =\int_{\A}[  U(s) J(z,h)] dz+\int_{\A}[V(t)] dz+\int_{\A}X dz
\\&\ge\int_{\A}[  U(s) J(z,h_\circ)] dz+\int_{\A}[V(t)] dz+\int_0^{2\pi}\int_1^r \Gamma(t, |h(te^{i\theta}|))\frac{|h|_N}{|z|}dt d\theta
\\&\ge \int_{\A}[  U(s) J(z,h_\circ)] dz+\int_{\A}[V(t)] dz+\int_0^{2\pi}\int_1^r\frac{\partial\mathcal{A}(t, |h(te^{i\theta}|))}{\partial t} dt d\theta
\\&= \int_{\A}[  U(s) J(z,h_\circ)] dz+\int_{\A}[V(t)] dz+2\pi (\mathcal{A}(r, R)-\mathcal{A}(1, 1))
\\&=\int_{\A}\mathfrak{D}[h_\circ]dz.
\end{split}
\]

\subsection{Proof of convexity case:  $H''_q(t)\ge 0$ ($\mathbf{c}>1$)}

If  $F=H^{-1}$ then from \eqref{hsecond}  we obtain.
$$F''(\tau)=-\frac{F'(\tau)^2 \left(- F(\tau)+\mathbf{c}^{-2}\tau F'(\tau)\right) \left(\tau+\gamma F(\tau) F'(\tau)\right)}{\tau F(\tau) \left(\gamma  F(\tau)+\mathbf{c}^{-2}\tau{F}'(\tau)^3\right)}.$$ Since $\mathbf{c}>1$, we have $1/\mathbf{c}<1$, so this case reduces to the concavity case for the inverse mappings: $\mathcal{F}(\B, \A)$.

\subsection{Proof of uniqueness}
The proof is similar to the proof in \cite{arma}. We start with the conditions \eqref{247}, \eqref{248}, \eqref{249} and \eqref{250}. We have also that $$b=\frac{\mathbf{b}}{\mathbf{a}}=b_\ast.$$ Hence $|h_N | = A$ and $|h_T | = a_* |h_N | = a_*A$. It follows from \eqref{250} that $h_T \overline{h_N }\in i \R_+$. Combining and using Lemma~\ref{71} we obtain
\begin{equation}\label{270}
\begin{split}
J (z, h) &= \Im(h_T\overline{ h_N}) = h_T \overline{ h_N} = |h_T | |h_N | = a_*(t)A^2(t,\tau)
\\&= a_* (|z|) {A}^2 (|z|, |h(z)|) = a_* (|z|) {A}^2 (|z|, H(|z|)). \end{split}
\end{equation} Thus, all extremal deformations must satisfy
\begin{equation}\label{271}h \overline{ h_N}\in \R \end{equation}
\begin{equation}\label{272}i\overline{h} h_T \in  \R \end{equation}
\begin{equation}\label{273}J (z, h)\text{ is a function in }|z| .\end{equation}
Now we deduce what we need from the following simple proposition
\begin{proposition}\cite{arma}
Let $h : A \into
 \R^2\setminus\{0\}$ be a homeomorphism of Sobolev class
$W ^{1,1}(A, \R^2)$ that satisfies the conditions \eqref{271}, \eqref{272} and \eqref{273}. Then $h$ is radial.
\end{proposition}

We finish this paper by stating the following conjecture. Namely we believe that the convexity hypothesis is not essential in Theorem~\ref{comtotal}.

 \begin{conjecture}\label{conjecture}
The total combined energy integral $\mathfrak{E}: \mathcal{F}(\A,\B)\to \mathbf{R}$,  attains its minimum  for a radial mapping $h_\circ$, without assuming any  convexity hypothesis.
\end{conjecture}

\end{document}